\documentclass{amsart}

\usepackage{amsfonts}
\usepackage{amsmath}
\usepackage{amssymb}
\usepackage{amsthm}
\usepackage{color}
\usepackage{graphicx}
\usepackage{xcolor}
\usepackage{xspace}
\usepackage{enumerate}

\newtheorem{theorem}{Theorem}
\newtheorem{lemma}[theorem]{Lemma}

\newtheorem{corollary}[theorem]{Corollary}

\theoremstyle{definition}
\newtheorem{definition}[theorem]{Definition}
\newtheorem{remark}[theorem]{Remark}

\newcommand{\ie}{i.{}e.{}\xspace} 
\newcommand{\R}{\mathbb{R}} 
\newcommand{\bd}{\mathbf{d}} 
\newcommand{\bm}{\mathbf{m}} 
\newcommand{\bn}{\mathbf{n}} 
\newcommand{\set}[2]{\left\{#1\,:\,#2\right\}} 
\DeclareMathOperator{\supp}{supp} 
\newcommand{\eqdef}{\mbox{\,\raisebox{0.2ex}{\scriptsize\ensuremath{\mathrm:}}\ensuremath{=}\,}} 
\newcommand{\defn}[1]{{\sl{\color{blue}#1}}} 
\newcommand{\ssm}{\smallsetminus} 
\newcommand{\del}[2]{#1-#2} 
\newcommand{\link}[2]{\operatorname{lk}_{#1}(#2)} 
\DeclareMathOperator{\diam}{diam} 
\newcommand{\one}{1\!\!1} 
\DeclareMathOperator{\rank}{rank} 
\DeclareMathOperator{\corank}{corank} 

\begin{document}

\title[Obstructions to weak decomposability for simplicial polytopes]{Obstructions to weak decomposability \\ for simplicial polytopes}

\author{Nicolai H\"ahnle}
\address[NH]{Technische Universit\"at Berlin}
\email{haehnle@math.tu-berlin.de}
\urladdr{}

\author{Steven Klee}
\address[SK]{University of California, Davis}
\email{klee@math.ucdavis.edu}
\urladdr{http://www.math.ucdavis.edu/~klee}

\author{Vincent Pilaud}
\address[VP]{CNRS \& LIX, \'Ecole Polytechnique, Palaiseau}
\email{vincent.pilaud@lix.polytechnique.fr}
\urladdr{http://www.lix.polytechnique.fr/~pilaud/}

\thanks{
SK was partially supported by  NSF VIGRE grant DMS-0636297.
VP was partially supported by grant MTM2011-22792 of the Spanish Ministerio de Ciencia e Innovaci\'on and by European Research Project ExploreMaps (ERC StG 208471).
}

\begin{abstract}
Provan and Billera introduced notions of (weak) decomposability of simplicial complexes as a means of attempting to prove polynomial upper bounds on the diameter of the facet-ridge graph of a simplicial polytope.
Recently, De Loera and Klee provided the first examples of simplicial polytopes that are not weakly vertex-decomposable.
These polytopes are polar to certain simple transportation polytopes.
In this paper, we refine their analysis to prove that these $d$-dimensional polytopes are not even weakly $O(\sqrt{d})$-decomposable.  As a consequence, (weak) decomposability cannot be used to prove a polynomial version of the Hirsch conjecture.
\end{abstract}

\maketitle

\section{Introduction}

The Hirsch conjecture, which was originally proposed by Hirsch in a letter to Dantzig in 1957, stated that the graph of a simple $d$-dimensional polytope with~$n$ facets has diameter at most~$n-d$.
Over the past 55 years, the Hirsch conjecture has motivated a tremendous amount of research on diameters of polytopes.
We do not attempt to recount this work here, but instead we point out excellent surveys of De Loera~\cite{DeLoera}, Kim and Santos~\cite{KimSantos}, and Todd~\cite{Todd}.
The most remarkable result in this realm is Santos's counterexample to the Hirsch conjecture~\cite{Santos} and its recent improvement by Matschke, Santos, and Weibel~\cite{Hirsch-Improved}.

The recent counterexamples to the Hirsch conjecture have answered one question and posed many more.
One of the most notable problems is the polynomial Hirsch conjecture, which states that the diameter of the graph of a simple $d$-dimensional polytope with $n$ facets is bounded by a polynomial function of~$n$ and~$d$.
In this paper, we work in the polar setting: we consider a simplicial $d$-dimensional polytope~$P$ on~$n$ vertices and seek bounds on the diameter of the facet-ridge graph of~$P$, which has a node for each facet of~$P$ and an edge between two nodes if their corresponding facets intersect along a ridge (a codimension-one face).

Provan and Billera~\cite{ProvanBillera} defined a notion of (weak) decomposability for an arbitrary simplicial complex and showed that the facet-ridge graph of a (weakly) $k$-decomposable simplicial complex has diameter bounded by a linear function of the number of $k$-dimensional faces of~$\Delta$.
Since the number of $k$-dimensional faces in an arbitrary simplicial complex on~$n$ vertices is bounded by~${n \choose k+1}$, the diameter of a (weakly) $k$-decomposable simplicial complex is bounded by a polynomial of degree~$k+1$ in its number~$n$ of vertices and its dimension~$d$.

One approach to proving the polynomial Hirsch conjecture would be to show that there is a fixed integer~$k$ such that all simplicial polytopes are (weakly) $k$-decomposable.
In this paper, we show that this approach will not work.
Specifically, we show that for any fixed integer~$k$, there is a simplicial $d$-dimensional polytope that is not weakly $k$-decomposable provided $d \geq \frac{(k+3)^2}{2}$.
Despite this obstruction, these polytopes still satisfy the Hirsch conjecture.

The remainder of the paper is structured as follows.
In Section~\ref{section:background}, we give all relevant definitions related to simplicial complexes, their diameters, and (weak) decomposability.
Further, we discuss the family of (simple) transportation polytopes and their combinatorial properties.
In Section~\ref{section:non-decomposable}, we define a certain family of simple $2 \times n$~transportation polytopes whose polars are not weakly $k$-decomposable for any fixed integer $k$.

\section{Background and definitions}
\label{section:background}

\subsection{Simplicial complexes and decomposability}

A \defn{simplicial complex}~$\Delta$ on vertex set~$V$ is a collection of subsets~$\tau \subseteq V$ (called \defn{faces}) such that if~$\tau$ is a face of~$\Delta$ and~$\sigma \subseteq \tau$, then~$\sigma$ is also a face of~$\Delta$.
The \defn{dimension} of a face~$\tau \in \Delta$ is~$\dim \tau \eqdef |\tau|-1$, and the \defn{dimension} of~$\Delta$ is~$\dim \Delta \eqdef \max \set{\dim \tau}{\tau \in \Delta}$.
We say that~$\Delta$ is \defn{pure} if all of its \defn{facets} (maximal faces under inclusion) have the same dimension.

We will borrow terminology from matroid complexes to define the \defn{rank} of a subset~$S \subseteq V$ to be~$\rank(S) \eqdef \max\set{|\tau|}{\tau \subseteq S, \; \tau \in \Delta}$, and the \defn{corank} of~$S$ to be~$\corank(S) \eqdef |S|-\rank(S)$.
Alternatively, the rank of~$S$ is one more than the dimension of the restriction of~$\Delta$ to the vertices in~$S$.

Let~$\tau$ be a face of the simplicial complex~$\Delta$.
The \defn{deletion} of~$\tau$ in~$\Delta$ is the simplicial complex~$\del{\Delta}{\tau} \eqdef \set{\sigma \in \Delta}{\tau \not\subseteq \sigma}$.
The \defn{link} of~$\tau$ in~$\Delta$ is the simplicial complex~$\link{\Delta}{\tau} \eqdef \set{\sigma \in \Delta}{\sigma \cap \tau = \emptyset, \; \sigma \cup \tau \in \Delta}$.

Given a pure simplicial complex~$\Delta$ and facets~$\sigma, \tau \in \Delta$, the \defn{distance} from $\sigma$~to~$\tau$ is the length of the shortest path~$\sigma = \rho_0, \rho_1, \dots, \rho_t = \tau$ where the~$\rho_i$ are facets of~$\Delta$ and~$\rho_i$ intersects~$\rho_{i+1}$ along a \defn{ridge} (a codimension-one face) for all~$0 \leq i < t$.
The \defn{diameter} of a pure simplicial complex, denoted $\diam(\Delta)$, is the maximum distance between any two facets in $\Delta$.

One approach to establishing diameter bounds for pure simplicial complexes is to study decompositions that generalize the more well-known property of a shedding order in a natural way.
Provan and Billera~\cite{ProvanBillera} defined a notion of (weak) $k$-decomposability for pure simplicial complexes and showed that (weakly) $k$-decomposable simplicial complexes satisfy natural diameter bounds.

\begin{definition}[\protect{\cite[Definition~2.1]{ProvanBillera}}]
\label{def:decomposability}
Let $\Delta$ be a $(d-1)$-dimensional simplicial complex and let $0 \leq k \leq d-1$.
We say that $\Delta$ is \defn{$k$-decomposable} if it is pure and if either~$\Delta$ is a $(d-1)$-simplex or there exists a face~$\tau \in \Delta$  with $\dim(\tau) \le k$ (called a \defn{shedding} face) such that
\begin{enumerate}[(i)]
	\item the deletion~$\del{\Delta}{\tau}$ is $(d-1)$-dimensional and $k$-decomposable, and
	\item the link $\link{\Delta}{\tau}$ is $(d-|\tau|-1)$-dimensional and $k$-decomposable.
\end{enumerate}
\end{definition}

\begin{theorem}[\protect{\cite[Theorem~2.10]{ProvanBillera}}]
If~$\Delta$ is a $k$-decomposable $(d-1)$-dimensional simplicial complex, then $$\diam(\Delta) \le f_k(\Delta)-{d \choose k+1},$$ where $f_k(\Delta)$ denotes the number of $k$-dimensional faces in $\Delta$.
\end{theorem}

\begin{definition}[\protect{\cite[Definition~4.2.1]{ProvanBillera}}]
\label{def:weakecomposability}
Let $\Delta$ be a $(d-1)$-dimensional simplicial complex and let $0 \leq k \leq d-1$.
We say that $\Delta$ is \defn{weakly $k$-decomposable} if it is pure and if either~$\Delta$ is a $(d-1)$-simplex or there exists a face~$\tau \in \Delta$ of dimension~$\dim(\tau) \le k$ (called a \defn{shedding} face) such that the deletion~$\del{\Delta}{\tau}$ is $(d-1)$-dimensional and weakly $k$-decomposable.
\end{definition}

\begin{theorem}[\protect{\cite[Theorem~4.2.3]{ProvanBillera}}]
If~$\Delta$ is a weakly $k$-decomposable $(d-1)$-dimensional simplicial complex, then $$\diam(\Delta) \le 2f_k(\Delta).$$
\end{theorem}

In particular, any $0$-decomposable complex (also called \defn{vertex-decomposable}) satisfies the Hirsch conjecture while any weakly $0$-decomposable complex (also called \defn{weakly vertex-decomposable}) has diameter bounded by a linear function of its number of vertices.
Klee and Kleinschmidt~\cite[Proposition 6.3]{Klee-Kleinschmidt} pointed out that Lockeberg~\cite{Lockeberg} had constructed a $4$-dimensional polytope on $12$ vertices that was not vertex-decomposable (but still satisfied the Hirsch bound), while the non-Hirsch polytopes of Santos \cite{Santos} and Matschke-Santos-Weibel \cite{Hirsch-Improved} provide further examples of simplicial polytopes that are not vertex-decomposable.
In a recent paper~\cite{DeLoeraKlee}, De~Loera and the second author provided the first examples of simplicial polytopes that are not even weakly vertex-decomposable.
Their smallest counterexample is a $4$-dimensional polytope on $10$ vertices and $30$ facets, obtained as the polar of the intersection of the $5$-dimensional cube~$[-1,1]^5$ with the linear hyperplane of equation~$\sum_i x_i = 0$.

Recently, a great deal of effort has been put forth in trying to conjecture what the correct upper bound on polytope diameters should be.
Santos's non-Hirsch polytopes all have diameter~$(1+\epsilon)(n-d)$, while the best known diameter bounds are quasi-exponential in~$n$ and~$d$~\cite{Kalai-Kleitman} or linear in fixed dimension but exponential in~$d$~\cite{Barnette, Larman}.
In the hopes of establishing a polynomial diameter bound, it is natural to ask whether there is a constant~$k$ such that all simplicial polytopes are (weakly) $k$-decomposable.
We will answer this question in the negative by establishing that for any~$k \geq 0$, there exist simplicial polytopes that are not weakly $k$-decomposable.
Our counterexamples are motivated by the counterexamples introduced by De Loera and Klee \cite{DeLoeraKlee}, which arise as the polars of a certain family of simple transportation polytopes.  The necessary background on transportation polytopes is briefly presented below.

\subsection{Transportation polytopes}

Transportation polytopes are classical polytopes from optimization.
Here we recall their basic definitions and combinatorial properties and refer to classical surveys on the topic for more details and proofs~\cite{KleeWitzgall, YemelichevKovalevKravtsov}.

\begin{definition}
Fix two integers~$m,n \ge 1$, and two vectors~$\bm \in \R^m$ and~$\bn \in \R^n$.
The \defn{$m \times n$ transportation polytope} $P(\bm,\bn)$ is the collection of all non-negative matrices~$X \eqdef (x_{\mu,\nu}) \in \R^{m \times n}_{\ge 0}$ such that
$$\forall \mu \in [m], \quad \sum_{\nu \in [n]} x_{\mu,\nu} = \bm_{\mu} \qquad \text{and} \qquad \forall \nu \in [n],\quad \sum_{\mu \in [m]} x_{\mu,\nu} = \bn_{\nu}.$$
The vectors~$\bm$ and~$\bn$ are called the \defn{margins} of~$P(\bm,\bn)$.
\end{definition}

Intuitively, a point in the transportation polytope~$P(\bm,\bn)$ is an assignment of weights to be transported between~$m$ sources and~$n$ sinks on the edges of the complete bipartite graph~$K_{m,n}$ such that the total quantity that a source~$\mu$ provides corresponds to its supply~$\bm_{\mu}$ while the total quantity that a sink~$\nu$ receives corresponds to its demand~$\bn_{\nu}$.
The \defn{support} of a point~$X \eqdef (x_{\mu,\nu})$ of~$P(\bm,\bn)$ is the subgraph~$\supp(X) \eqdef \set{(\mu,\nu) \in K_{m,n}}{x_{\mu,\nu} > 0}$ of the complete bipartite graph~$K_{m,n}$.
The following theorem summarizes classical properties of transportation polytopes.
See~\cite{KleeWitzgall, YemelichevKovalevKravtsov} for proofs.

\begin{theorem}
\label{theo:propertiesTP}
The transportation polytope~$P(\bm,\bn)$, with margins~$\bm \in \R^m$ and $\bn \in \R^n$, has the following properties.

\begin{description}
\setlength{\itemsep}{1mm}
	\item[Feasability] $P(\bm,\bn)$ is non-empty if and only if~$\sum_{\mu \in [m]} \bm_{\mu} = \sum_{\nu \in [n]} \bn_{\nu}$.
	\item[Dimension] The dimension of~$P(\bm, \bn)$ is~$(m-1)(n-1)$, provided it is nonempty.
	\item[Non-degeneracy] $P(\bm,\bn)$ is nondegenerate (hence simple) if and only if there are no proper subsets~$\emptyset \subsetneq M \subsetneq [m]$ and~$\emptyset \subsetneq N \subsetneq [n]$ such that $\sum_{\mu \in M} \bm_{\mu} = \sum_{\nu \in N} \bn_{\nu}$. The support of the points of~$P(\bm,\bn)$ are then connected and spanning.
	\item[Vertices] A point of~$P(\bm,\bn)$ is a vertex of~$P(\bm,\bn)$ if and only if its support is a forest (\ie~contains no cycle). In particular, if~$P(\bm,\bn)$ is nondegenerate, a point of~$P(\bm,\bn)$ is a vertex if and only if its support is a spanning tree of~$K_{m,n}$.
	\item[Facets] All facets of~$P(\bm,\bn)$ are of the form~$F_{\mu,\nu} \eqdef \set{X \in P(\bm,\bn)}{x_{\mu,\nu} = 0}$ for some~$\mu \in [m]$ and~$\nu \in [n]$. Moreover, if~$mn > 4$, then the set~$F_{\mu,\nu}$ is a facet if and only if~$\bm_\mu + \bn_\nu < \sum_{\mu' \in [m]} \bm_{\mu'} = \sum_{\nu' \in [n]} \bn_{\nu'}.$
\end{description}
\end{theorem}

Optimizing transportation costs naturally gives rise to linear optimization problems on transportation polytopes, and thus leads to the question to evaluate the diameter of transportation polytopes.
Although this question has been largely studied in the literature, the Hirsch conjecture for transportation polytopes is still open.
On the one hand, Brightwell, van den Heuvel, and Stougie~\cite{BrightwellvdHeuvelStougie} showed that the diameter of any $m \times n$ transportation polytope is at most~$8(m+n-1)$, \ie eight times the Hirsch bound.
On the other hand, the precise Hirsch bound was only proven for~$2 \times n$ transportation polytopes by Kim~\cite[Section~3.5]{Kim} and for the family of signature polytopes, defined as follows by Balinski and Rispoli~\cite{BalinskiRispoli}.

\begin{definition}[\cite{BalinskiRispoli}]
\label{def:signaturePolytope}
Fix a vector~$\bd \in \R^m$ such that~$\sum_{\mu \in [m]} \bd_\mu = m+n-1$.
Let~$\mathcal{T}(\bd)$ denote the set of all spanning trees of the complete bipartite graph~$K_{m,n}$ where each vertex~$\mu \in [m]$ has degree~$\bd_\mu$.
An $m \times n$ transportation polytope~$P(\bm, \bn)$ is a \defn{$\bd$-signature polytope} if any tree of~$\mathcal{T}(\bd)$ is the support of a vertex of~$P(\bm,\bn)$.
\end{definition}

Controlling the feasible trees of signature polytopes, Balinski and Rispoli~\cite{BalinskiRispoli} prove that these polytopes satisfy the monotone Hirsch conjecture.
In particular, for any vector~$\bd \in \R^m$ such that~$\sum_{\mu \in [m]} \bd_\mu = m+n-1$, observe that the $m \times n$ transportation polytope~$P(\bm,\bn)$ with margins ${\bm \eqdef m\bd - (m+1)\one \in \R^m}$ and~$\bn \eqdef m\one \in \R^n$ is a signature polytope and thus satisfies the monotone Hirsch conjecture. These particular examples of signature polytopes were previously studied by Balinski~\cite{Balinski}.

\section{Simplicial polytopes that are not weakly $k$-decomposable}
\label{section:non-decomposable}

In this section, we study a certain family of simple $2 \times n$ transportation polytopes whose polars are not weakly $k$-decomposable for sufficiently large $n$.

\begin{definition}
\label{def:DeltaAB}
Fix integers~$a,b \geq 1$ and define~$\Delta(a,b)$ to be the boundary complex of the polar polytope to the $2 \times (a+b+1)$ transportation polytope~$P(\bm,\bn)$ with margins
$$\bm \eqdef (2a+1,2b+1) \in \R^2 \qquad \text{and} \qquad \bn  \eqdef (2,2,\dots,2) \in \R^{a+b+1}.$$
\end{definition}

According to Theorem~\ref{theo:propertiesTP}, $P(\bm,\bn)$ is a simple polytope of dimension~$a+b$ whose facets are precisely the sets~$F_{\mu,\nu} \eqdef \set{X \in P(\bm,\bn)}{x_{\mu,\nu} = 0}$, for~$\mu \in [2]$ and~$\nu \in [a+b+1]$.
For simplicity, we decompose the vertex set of $\Delta(a,b)$ into sets $U \eqdef \{u_1, \dots, u_{a+b+1}\}$ and $V \eqdef \{v_1, \dots, v_{a+b+1}\}$ where $u_\nu$ is the polar vertex to the facet $F_{1,\nu}$ and $v_\nu$ is the polar vertex to the facet $F_{2,\nu}$.
That is to say, the vertices of $U$ correspond to the coordinates of the top row of the $2 \times n$ contingency table defining $P(\bm, \bn)$ and the vertices of $V$ correspond to the coordinates of the bottom row.
Theorem~\ref{theo:propertiesTP} yields the following description of the facets of~$\Delta(a,b)$.

\begin{lemma}
The facets of~$\Delta(a,b)$ are precisely the sets~$A \cup B$ where $A \subseteq V$, $B \subseteq U$, $|A| = a$, $|B| = b$, and $A \cup B$ contains at most one element from each pair $\{u_\nu,v_\nu\}$.
\end{lemma}

\begin{proof}
Consider a facet~$F$ of $\Delta(a,b)$.
Let~$X \eqdef (x_{\mu,\nu})$ be the vertex of~$P(a,b)$ polar to~$F$, let~$A \eqdef F \cap V$, and let $B \eqdef F \cap U$.
If~$v_\nu \in A$, then $x_{2,\nu} = 0$, and thus $x_{1,\nu} = 2$.
Consequently, $2a+1 = \sum_\nu x_{1,\nu} \ge 2|A|$.
Similarly $2b+1 \ge 2|B|$.
Since $|A|+|B| = a+b$, we obtain that~$|A| = a$ and~$|B| = b$.
Finally, $F$ cannot contain both $u_\nu$ and $v_\nu$ since~$X$ cannot have $x_{1,\nu} = x_{2,\nu} = 0$.

Conversely, let~$A \subseteq V$ and~$B \subseteq U$ be such that $|A| = a$, $|B| = b$, and $A \cup B$ contains at most one element from each pair $\{u_\nu,v_\nu\}$.
Let~$\bar B \eqdef \set{\nu}{u_\nu \in B}$ and $\bar A \eqdef \set{\nu}{v_\nu \in A}$.
Then~$\bar A$ and~$\bar B$ are disjoint subsets of~$[a+b+1]$ and their complement~$[a+b+1] \ssm (\bar A \cup \bar B)$ has a unique element~$\bar \nu$.
Define~$X \eqdef (x_{\mu,\nu})$ by
$$x_{1,\nu} \eqdef \begin{cases} 1 & \text{if } \nu = \bar\nu \\ 2 & \text{if } \nu \in \bar A \\ 0 & \text{if } \nu \in \bar B \end{cases} \qquad \text{and} \qquad x_{2,\nu} \eqdef \begin{cases} 1 & \text{if } \nu = \bar\nu \\ 0 & \text{if } \nu \in \bar A \\ 2 & \text{if } \nu \in \bar B \end{cases}.$$
The point~$X$ is a vertex of~$P(\bm,\bn)$ since it is a point of~$P(\bm,\bn)$ whose support is a spanning tree.
Moreover, its polar facet in~$\Delta(a,b)$ is~$A \cup B$.
\end{proof}

\begin{remark}
We can alternatively describe the transportation polytope~$P(\bm,\bn)$ of Definition~\ref{def:DeltaAB} as the intersection~$P$ of the $(a+b+1)$-dimensional cube~$C \eqdef [0,2]^{a+b+1}$ with the hyperplane~$H$ of equation~$\sum_i x_i = 2a+1$.
The $2(a+b+1)$ facets of~$P$ are supported by the hyperplanes supporting the facets of the cube~$C$.
The $\frac{(a+b+1)!}{a!b!}$ vertices of~$P$ are the intersections of the hyperplane~$H$ with the edges of the cube~$C$ at distance~$2a$ from the origin.
Equivalently, $P$ is the Minkowski sum of the hypersimplices~$\triangle_a + \triangle_{a+1}$, where the \defn{hypersimplex}~$\triangle_a$ is the intersection of the cube~$[0,1]^{a+b+1}$ with the hyperplane~$\sum_i x_i=a$.
\end{remark}

According to the above-mentioned results on~$2 \times n$ transportation polytopes \cite[Section~3.5]{Kim} and on signature polytopes~\cite{Balinski, BalinskiRispoli}, the transportation polytope~$P(\bm,\bn)$ satisfies the (monotone) Hirsch conjecture.
In contrast, the results of~\cite[Theorem 3.1]{DeLoeraKlee} imply that the boundary complex~$\Delta(a,b)$ of the polar of~$P(\bm,\bn)$ is not weakly vertex-decomposable when~$a,b \geq 2$.
We now refine the analysis of the weak decompositions of~$\Delta(a,b)$ to derive the following theorem.

\begin{theorem}
\label{theo:nonWeaklyKDecomposable}
If~$\left(\frac{k+3}{2}\right)^2 \le \min(a,b)$, then the simplicial complex~$\Delta(a,b)$ of Definition \ref{def:DeltaAB} is not weakly $k$-decomposable.
\end{theorem}

\begin{corollary}
There exist $d$-dimensional simplicial polytopes whose boundary complexes are not weakly $O(\sqrt{d})$-decomposable, hence not strongly $O(\sqrt{d})$-decomposable.
\end{corollary}

To prove Theorem~\ref{theo:nonWeaklyKDecomposable}, we will suppose by way of contradiction that $\Delta(a,b)$ is weakly $k$-decomposable with shedding sequence $\tau_1, \tau_2, \dots, \tau_t$.
We set $\Delta_0 \eqdef \Delta(a,b)$ and ${\Delta_i \eqdef \del{\Delta_{i-1}}{\tau_i}}$, for~$i \geq 1$.
Each $\Delta_i$ is a pure simplicial complex of dimension $a+b-1$ and $\Delta_t$ consists of a single $(a+b-1)$-simplex.
Our goal now is to understand how shedding each of the faces~$\tau_i$ affects certain subsets of vertices of~$\Delta(a,b)$.
With this motivation, we define a family of functions
$$\varphi_i: {U \choose \leq b+1} \cup {V \choose \leq a+1} \rightarrow \mathbb{N}, $$
by setting $\varphi_i(S)$ to be the corank of the set of vertices $S$ in the complex $\Delta_i$, \ie
$$\varphi_i(S) \eqdef |S| - \max \set{|\tau|}{\tau \subset S, \; \tau \in \Delta_i}.$$
Here we use the notation~${W \choose \leq w}$ to denote the collection of subsets of the set~$W$ of size at most~$w$.
The following properties of the functions $\{\varphi_i\}_{i \in [t]}$ are immediate.

\begin{lemma}\label{lem:properties-of-phi-functions}
\begin{enumerate}
\item $\varphi_0(S) \leq 1$ for all $S$.
\item There exists a set $S$ such that $\varphi_t(S) \geq 2$.
\item If $\tau_i \cap U \neq \emptyset$ and $\tau_i \cap V \neq \emptyset$, then $\varphi_i(S) = \varphi_{i-1}(S)$ for all $S$.
\item If $\tau_i \subseteq U$, then $\varphi_i(S) = \varphi_{i-1}(S)$ for all $S \in {V \choose \leq a+1}$.
\item $\varphi_i(S) \leq 1$ if and only if there is some $v \in S$ such that $S \ssm v \in \Delta_i$.
\end{enumerate}
\end{lemma}

The following lemma will be essential to the proof of Theorem \ref{theo:nonWeaklyKDecomposable}.

\begin{lemma}\label{lem:hitting-sets}
Let $\mathcal{X}$ be a collection of sets, each of which has size at most $k+1$, and suppose $\bigcap_{X \in \mathcal{X}}X =\emptyset$.
Then there is a sub-collection $\mathcal{Y} \subseteq \mathcal{X}$ such that
$$\bigcap\nolimits_{X \in \mathcal{Y}} X = \emptyset \qquad \text{and} \qquad \left|\bigcup\nolimits_{X \in \mathcal{Y}}X\right| \leq \left(\frac{k+3}{2}\right)^2.$$
\end{lemma}

\begin{proof}
Let $\mathcal{Y}$ be a minimal collection of sets of $\mathcal{X}$ such that $\bigcap_{X \in \mathcal{Y}}X = \emptyset$  (\ie any proper sub-collection of sets in $\mathcal{Y}$ has a non-empty intersection).

For each $X \in \mathcal{Y}$, there is an element $f(X) \in \bigcap_{Y \in \mathcal{Y} \ssm \{X\}}Y$ by our assumption that $\mathcal{Y}$ is minimal.
Furthermore, $f(X) \notin X$, since otherwise $\bigcap_{Y \in \mathcal{Y}}Y \neq \emptyset$.
Hence the elements $\set{f(X)}{X \in \mathcal{Y}}$ are all distinct.
This means any $X \in \mathcal{Y}$ contains each of the elements $f(Y)$ with $Y \in \mathcal{Y} \ssm \{X\}$ and at most $k + 2 - |\mathcal{Y}|$ other elements.
Thus
$$\left| \bigcup\nolimits_{X \in \mathcal{Y}}X \right| \; \leq \; |\mathcal{Y}| +  |\mathcal{Y}|\left(k + 2 - |\mathcal{Y}|\right) \; = \; |\mathcal{Y}|\left(k + 3 - |\mathcal{Y}|\right) \; \leq \; \left(\frac{k+3}{2}\right)^2,$$
since the quadratic function~$y \mapsto y(\alpha-y)$ is maximized when $y = \frac{\alpha}{2}$.
\end{proof}

\begin{remark}
Observe that the bound in the previous lemma is tight.
Consider all $\lfloor\frac{k+1}{2}\rfloor$-element subsets of a $(\lfloor\frac{k+1}{2}\rfloor+1)$-element set.
Add to each of these sets~$\lceil\frac{k+1}{2}\rceil$ new elements (all additional elements are distinct).
The intersection of the resulting~$\lfloor\frac{k+1}{2}\rfloor+1$ sets is empty, but the intersection of any proper sub-collection is non-empty. Moreover, their union has cardinality $(\lfloor\frac{k+1}{2}\rfloor+1)(\lceil\frac{k+1}{2}\rceil+1)$.
\end{remark}

We are now ready to prove Theorem \ref{theo:nonWeaklyKDecomposable}.

\begin{proof}[Proof of Theorem \ref{theo:nonWeaklyKDecomposable}]

By Lemmas~\ref{lem:properties-of-phi-functions}(1) and~\ref{lem:properties-of-phi-functions}(2), we can consider the smallest value $i \in [t]$ for which there exists some~$S$ with $\varphi_i(S) \geq 2$.
By Lemma~\ref{lem:properties-of-phi-functions}(3), the corresponding shedding face $\tau_i$ is a subset of either $U$ or $V$.
Without loss of generality, we may assume that $\tau_i \subseteq U$.

Let $\mathcal{X}$ denote the collection of shedding faces $\tau \in \{\tau_1, \dots, \tau_i\}$ such that $\tau \subseteq S$.
Observe that $\bigcap_{\tau \in \mathcal{X}}\tau = \emptyset$.
Indeed, otherwise if there is an element $u \in \bigcap_{\tau \in \mathcal{X}}\tau$, then $S \ssm u$ does not contain any set $\tau \in \mathcal{X}$.
This would mean that $S \ssm u \in \Delta_i$ and hence $\varphi_i(S) \leq 1$ by Lemma~\ref{lem:properties-of-phi-functions}(5).
According to Lemma \ref{lem:hitting-sets}, there exists a sub-collection $\mathcal{Y} \subseteq \mathcal{X}$ of shedding faces such that
$$\bigcap\nolimits_{\tau \in \mathcal{Y}}\tau = \emptyset \qquad \text{and} \qquad \left| \bigcup\nolimits_{\tau \in \mathcal{Y}}\tau \right| \le \left( \frac{k+3}{2} \right)^2 \le b.$$
Replacing~$S$ by~$\bigcup_{\tau \in \mathcal{Y}}\tau$, we can thus assume that~$\varphi_i(S) \geq 2$ and~$|S| \leq b$.

Since~$|S| \le b$, we can choose a subset $T \subseteq V$ of $a+1$ vertices of~$\Delta(a,b)$ such that $S \cup T$ contains at most one element from each pair~$\{u_j,v_j\}$.
By our choice of $i$ and Lemma \ref{lem:properties-of-phi-functions}(4), $\varphi_i(T) = \varphi_{i-1}(T) \leq 1$.
Thus there is a subset $A \subseteq T$ with $|A| = a$ such that $A \in \Delta_i$.

We claim that $A$ is not contained in any face of size $a+b$ in $\Delta_i$, which will contradict our assumption that $\Delta_i$ is pure of dimension $a+b-1$.
Suppose that there is a subset $B \subseteq U$ with $|B| = b$ such that $A \cup B$ is a facet of $\Delta_i$.
Then $B \cap S$ is a face of $\Delta_i$ and $|B \cap S| \geq |S|-1$, meaning $\varphi_i(S) \leq 1$, which contradicts our choice of $S$.
\end{proof}

\section*{Acknowledgments}

We are grateful to Jes\'us De Loera and Paco Santos for a number of helpful and insightful conversations during the cultivation of the ideas that motivated this paper.
The main ideas behind this paper we discovered during the Triangulations workshop at the Oberwolfach MFO during May 2012.
We would like to extend our sincere thanks to the organizers -- William Jaco, Frank Lutz, Paco Santos, and John Sullivan -- for the invitation to attend the workshop.
Nicolai H\"ahnle and Steven Klee also thank the Oberwolfach Leibniz Grant for providing travel support for the workshop.

\bibliographystyle{alpha}
\bibliography{decomposabilityTransportationPolytopes}

\end{document}